\documentclass[12pt]{article}
\usepackage[english]{babel}
\usepackage{amsmath,amsthm,amsfonts,amssymb,epsfig,verbatim}
\usepackage[left=1in,top=1in,right=1in]{geometry}
\usepackage{pgf,tikz}
\usetikzlibrary{arrows}

\newtheorem{thm}{Theorem}[section]
\newtheorem{lem}[thm]{Lemma}
\newtheorem{prop}[thm]{Proposition}

\newcommand{\E}{\mathbb{E}}

\newcommand{\R}{\mathbb{R}}
\newcommand{\Z}{\mathbb{Z}}

\newcommand{\N}{\mathbb{N}}

\newcommand{\e}{\varepsilon}

\def\<{\langle}
\def\>{\rangle}

\newcommand{\Pro}{\ensuremath{\mathbb{P}}}

\newcommand{\Var}{\mathop{\mathrm{Var}}\nolimits}
\newcommand{\bb}{\mathbf} 

\theoremstyle{usual}
\newtheorem{theorem}{Theorem}[section]
\newtheorem{assumption}{Assumption}[section]

\newtheorem{lemma}{Lemma}[section]
\newtheorem{proposition}{Proposition}[section]

\newtheoremstyle{likedef}
  {}%
  {}%
  {}%
  {\parindent}%
  {\bfseries}%
  {.}%
  {.5em}%
  {}%

\theoremstyle{likedef}
\newtheorem{definition}{Definition}[section]
\newtheorem{remark}{Remark}

\numberwithin{equation}{section}

\begin{document}

\title{The scaling relation $\chi = 2 \xi - 1$ for directed polymers in a random environment}

\author{Antonio Auffinger\thanks{Mathematics Department, University of Chicago, 5734 S. University Avenue, Chicago, IL 60637. Email:auffing@math.uchicago.edu.}
\and
Michael Damron\thanks{Mathematics Department, Princeton University, Fine Hall, Washington Rd., Princeton, NJ 08544. Email: mdamron@math.princeton.edu;
Research funded by an NSF Postdoctoral Fellowship and NSF grants DMS-0901534 and DMS-1007626.}}
\date{\today}
\maketitle

\begin{abstract} We prove the scaling relation $\chi = 2\xi -1$ between  the transversal exponent $\xi$ and the fluctuation exponent $\chi$ for directed polymers in a random environment in $d$ dimensions. The definition of these exponents is similar to that proposed by S. Chatterjee in \cite{Sourav} in first-passage percolation. The proof presented here also establishes the relation in the zero temperature version of the model, known as last-passage percolation.
\end{abstract}

\section{Introduction}

This paper is about Directed Polymers in a Random Environment. In this model, we place non-negative, independent, identically distributed random variables $(\tau_e)$, one at each nearest neighbor edge of $\mathbb Z^{d}$. For $\bb u$, $\bb v$ vertices of $\mathbb Z^d$, a directed path from $\bb u$ to $\bb v$ is a sequence of vertices $(\bb v_k)_{k= 0}^{n}$, and nearest neighbor edges $e_k=( \bb v_k, \bb v_{k+1} )$, $k=0, \ldots, n-1$ such that $\bb v_0 =\bb u$, $\bb v_n = \bb v$ and the coordinates of the $\bb v_k$'s are non-decreasing in $k$.

Given $\beta>0$ we define the partition function from $\bb u$ to $\bb v$ at inverse temperature $\beta$ as   
 \begin{equation*}
  Z^{\beta}(\bb u,\bb v) = \sum_{\gamma: \bb u \rightarrow \bb v} \exp(- \beta \tau(\gamma))\ ,
  \end{equation*}
where the sum runs over all directed paths from $\bb u$ to $\bb v$ and $\tau(\gamma) = \sum_{e \in \gamma} \tau_e$. 
Note that to have a non-empty collection of directed paths one needs the coordinates of the final point to be greater than or equal to those of the initial point. We will write this condition as $\bb u \leq \bb v$. We then extend the partition function to $\mathbb{R}^d$ in the natural way: if $\bb u \in \mathbb{R}^d$ then write $[\bb u]$ for the unique lattice point such that $\bb u \in [\bb u] + [-1/2,1/2)^d$. We then define $Z^\beta(\bb u, \bb v) = Z^\beta([\bb u], [\bb v])$. Associated to $Z^\beta(\bb u, \bb v)$ is the random probability measure 
\[
\mu_{\bb u,\bb v}(\gamma) = \frac{1}{Z^{\beta}(\bb u,\bb v)} \, \exp ( -\beta \mathcal \tau(\gamma) )\ .
\]

In this paper we will study the relation between three exponents. The first one, denoted by $\chi$, measures the growth of the variance of the partition function $Z^\beta(\bb 0, n \bb e)$, where
\[
\bb e = (1, \ldots, 1) \in \Z^d\ ,
\]
as $n$ goes to infinity. The second, denoted by $\xi$, measures the transversal fluctuations of a typical path sampled from $\mu^\beta_{\bb 0, n\bb e}$. The third, denoted by $\kappa$, measures the curvature of the limiting free energy in the direction $\bb e$. We will show that these exponents are related by
\begin{equation}\label{eqAD}
\chi = \kappa \xi- (\kappa-1)\ .
\end{equation}
 
This scaling relation is now known for an undirected zero temperature version of the model that we consider here (first-passage percolation) (\cite{Sourav}, see also \cite{AD2}). The directed zero temperature case can be proved by methods similar to those presented here (See Remark~\ref{rem2}). The actual values of $\chi$ and $\xi$ are known for certain ``exactly solvable" models in two dimensions  (see \cite{BCD, Joha, Joha2, Sepa} for instance). For these models in a appropriate sense, $\xi = 2/3$, $\kappa =2$ and $\chi = 1/3$ and therefore \eqref{eqAD} holds.

It is conjectured that in any dimension, under mild assumptions on the distribution of the $\tau_e$'s, $\kappa =2$. In this case, \eqref{eqAD} becomes the famous KPZ scaling relation (see \cite{KPZ}):
\begin{equation}\label{eqKPZ}
\chi = 2 \xi -1\ .
\end{equation}

We will define these exponents in Section \ref{sec: exponents}, where we also state our main result. First, we will state the shape theorem for the free energy. This theorem is the analogue of the classical shape theorem proved by Richardson \cite{Richardson} in the the Eden model and then by Cox-Durrett \cite{CDurrett} for first-passage percolation models. The shape theorem was extended to \emph{directed} percolation models by Martin in \cite{JMartin03}. Our proofs follow their ideas with minor modifications and are presented in Appendix~\ref{appendix}. Let $| \cdot |_1$ denote the $\ell_1$ norm in $\Z^d$. For $\bb x \in \Z_+^d := \{\bb z = (z_1, \ldots, z_d) \in \mathbb{Z}^d :  z_i \geq 0 \text{ for all } i\}$, define the free energy as
\begin{equation}\label{eq:freeenergy}
F(\bb 0, \bb x) = -\frac{1}{\beta} \log \frac{Z^\beta (\bb 0,  \bb x)}{d^{| \bb x|_1}}\ .
\end{equation}
(The factor $d^{-|\bb x|_1}$ is present to force $F(\bb 0, \bb x)\geq 0$.)

We prove the following basic properties in Appendix~\ref{appendix}. They are analogous to ones proved for directed last-passage percolation \cite{JMartin03}. 

\begin{proposition}\label{prop:pro1} If $\mathbb{E} \tau_e < \infty$ then for all $\mathbf x, \mathbf y \in \mathbb \R^d_+$,
\begin{enumerate}
\item the following limit exists a.s. and in $L_1$:
\[
\lim_{n\rightarrow \infty} \frac{1}{n} F(\bb 0, n \bb x) =: f(\bb x) < \infty\ .
\]

\item $f$ is nonnegative. Furthermore,
\begin{equation}\label{eq: infbound}
\inf_{\bb x \in \mathbb{R}_+^d \setminus \{\bb 0\}} \frac{f(\bb x)}{|\bb x|_1} > 0 
\end{equation}
if and only if $\mathbb{P}(\tau_e = 0) <1$.
\item $f$ is positive homogenous; that is, for any $\lambda \geq 0$, $f(\lambda \bb x) = \lambda f(\bb x)$. 

\item $f$ is invariant under permutation of the cooordinates.
\item $f(\bb x+ \bb y) \leq f(\bb x) + f(\bb y)$.
\item $f$ is continuous.

\end{enumerate}
\end{proposition}


The function $f$ will be called the limiting free energy.  We set 
\[
B_t =  \{ \bb x \in \R^{d}_+ : F(\bb 0, \bb x) \leq t\} \text{ and } B= \{ \bb x \in \R^{d}_+ : f(\bb x) \leq 1\}\ .
\]
Note that by the above proposition, $B$ is compact and convex. The shape theorem is then the following:

\begin{proposition} \label{shapethm} 
If $\mathbb{E} \tau_e^{d+\alpha} < \infty$ for some $\alpha>0$ and $\mathbb{P}(\tau_e=0) <1$, then for any $\e >0$ 
\[
\Pro\bigg( (1-\e) B \subseteq  \frac{B_t}{t} \subseteq (1+\e)B \quad \text{for all sufficiently large} \; t\bigg) =1\ .
\]
\end{proposition}
We prove Propositions \ref{prop:pro1} and \ref{shapethm} in Appendix~\ref{appendix}.

\subsection{Exponents and main result}\label{sec: exponents}
We will now rigorously define the three exponents mentioned above.

Let $|\cdot|$ denote the Euclidean norm in $\R^{d}$.  Our main assumption on the limiting free energy is the following curvature requirement in the diagonal direction.
\begin{assumption}\label{assumption1} There exists a positive number $\kappa$ and positive constants $C_1, C_2, \e$ such that if $\bb z \cdot \bb e = 0$  and $|\bb z|<\e$  then 
\begin{equation}\label{eqassumption}
C_1 |\bb z|^{\kappa} \leq |f(\bb e+\bb z)-f(\bb e)| \leq C_2 |\bb z|^{\kappa}\ . \quad
\end{equation}
\end{assumption}
\begin{remark}We fixed the direction $\mathbf{e}$ to simplify notation. All theorems can be extended to any direction where the analogue of \eqref{eqassumption} holds. It is worth noting that it is always possible to find directions (possibly different) where the lower and upper bounds of \eqref{eqassumption} hold with $\kappa = 2$. (See for instance \cite[Section 5]{Sourav}.) 
\end{remark}

\begin{definition}\label{def1} The number $\kappa$ that satisfies \eqref{eqassumption} is called the curvature exponent of the polymer model in the diagonal direction.
\end{definition}

We now define the other two exponents.  Given $\bb x \in \R^{d}$ we set $L(\bb x)$ to be the line segment in $\R^d$ that interpolates between $\bb 0$ and  $ \bb x$. For any $r>0$, we define the cylinder of radius $r$ between $\bb 0$ and  $ \bb x$ as the set
 \[
 C_{\bb x}[r] := \left\{ \bb z \in \Z^d:  \inf_{\bb w \in L (\bb x)} |\bb z - \bb w|< r \right\}\ .
 \]
We say that a nearest neighbor path $\gamma$ is in the cylinder $C_{\bb x}[r]$ if all vertices of $\gamma$ lie in $C_{\bb x}[r]$.

\begin{definition}\label{def2}
The transversal exponent $\xi_a$ is the smallest real number such that for any $\xi' > \xi_a$  there exist $\alpha, \delta >0$ such that for all $n$ \begin{equation}
\Pro\bigg( \mu_{\bb 0, n\bb e} (\gamma \in C_{n\bb e}[n^{\xi'}])<1-\frac{1}{n^{1+\alpha}}\bigg) \leq e^{- n^\delta}\ .
\end{equation}
\end{definition}

\begin{definition}\label{def2b}
The transversal exponent $\xi_b $ is defined as 
\begin{equation}
\xi_b = \inf \bigg \{ \xi: \forall \; \e > 0, \quad \Pro \bigg( \mu_{\bb 0, n\bb e} (\gamma \in C_{n\bb e}[n^{\xi}] )> 1- \e \bigg)\rightarrow 1 \bigg \} \ .
\end{equation}
\end{definition}

Roughly speaking, the exponent $\xi$ is such that a typical polymer path of length $n$ deviates from the straight line by a distance of order $n^\xi$. Definition~\ref{def2} guarantees that the path is inside any cylinder of radius $n^{\xi'}$ for $\xi'>\xi_a$, while Definition~\ref{def2b} guarantees that a cylinder of radius $n^{\xi''}$ for $\xi'' < \xi_b$ is not large enough to contain the path. Note that trivially $0 \leq \xi_b \leq \xi_a \leq 1$.

We will need to define two fluctuation exponents. 

\begin{definition}\label{def3} The fluctuation exponent $\chi_a$ is defined as the smallest number such that for any $\chi'>\chi_a$,  there exists $\alpha >0$ such that
\begin{equation}\label{above}
\sup_{\bb v \in \mathbb{Z}_+^d \setminus \{ \bb 0\}} \E \exp\bigg( \alpha \frac{| F (\bb 0,\bb v ) - \mathbb{E}F(\bb 0, \bb v)|}{|\bb v |_1^{\chi'}}  \bigg) < \infty\ . 
 \end{equation}
\end{definition}

Definition~\ref{def3} says that the collection of random variables $\left(\frac{|F(\bb 0, \bb v) - \mathbb{E} F(\bb 0, \bb v)|}{|\bb v|_1^{\chi'}}\right)_{\bb v \in \mathbb{Z}_+^d \setminus \{\bb 0\}}$ is exponentially tight. It is known by the work of Piza \cite[Proposition~1(c)]{Piza} (see also \cite[Equation~(1.15)]{Kesten93}) that this holds for $\chi'=1/2$ if one assumes finite exponential moments for the distribution of $\tau_e$. The next definition guarantees that the variance of $F(\bb 0, \bb v)$ is not significantly smaller than $|\bb v|^{2\chi_b}$.

\begin{definition}\label{def4} The fluctuation exponent $\chi_b$ is defined as the largest number such that for any $\chi''<\chi_b$
\begin{equation}\label{below}
\inf_n \frac{\Var(F (\bb 0, n \bb e))}{n^{2\chi''}} >0\ .
 \end{equation}
\end{definition}

Our main result in this paper is the following. 
\begin{theorem}\label{theorem1} Assume that the polymer model has exponents as in definitions \ref{def1}-\ref{def4} with $\chi := \chi_a = \chi_b$ and $\xi:= \xi_a = \xi_b$. Then 
\begin{equation} \label{KPZu}
\chi = \kappa \xi - (\kappa -1)\ .
\end{equation}  
\end{theorem}

We finish this section with a few remarks.

\begin{remark}\label{rem2}
The directed zero temperature case, commonly called last-passage percolation, can be analyzed in the same way (and even with the same proof) as what is given here. The only difference is that we must make the assumption $\mathbb{P}(\tau_e = S) < 1$, where $S$ is the supremum of the distribution of $\tau_e$. In particular, one can show that under the assumption of existence of exponents analogous to above, one has the relation $\chi = \kappa \xi - (\kappa - 1)$. Equation \eqref{KPZu} has been shown to hold for some definition of exponents in certain ``exactly solvable'' cases \cite{Joha}. For more information on exact solvable models the reader is invited to check the survey \cite{Corwin} and the references therein. 
\end{remark}

\begin{remark}\label{rem3}
For a log-gamma distribution on edge-weights in dimension 2, Sepp\"al\"ainen \cite{Sepa} has explicitly derived the limiting shape for the free energy. Consequently it can be verified that the exponent $\kappa$ equals 2 in this case. 
\end{remark}

\begin{remark}
Equation \eqref{KPZu} is trivially true when the environment is not present. Indeed,  for $\beta = 0$ the polymer path is roughly a simple random walk and therefore $\chi = 0$ and $\xi = 1/2$. In two dimensions, if $\beta$ scales to zero as a function of $n$ as $\beta=cn^{-1/4}$, it is also known that \eqref{KPZu} holds with $d=1$ with $\chi = 0,~ \xi = 1/2$. Interestingly, in this case, the fluctuations do not decouple from the random environment and the polymer path has non-trivial scaling limit \cite{AKQ}. Equation  \eqref{KPZu} also holds for directed polymers in thin cylinders, for directions asymptotically close to a coordinate axis \cite{ABC}.
\end{remark}

The rest of this manuscript is organized as follows. In Section 2, we prove the upper  bound $\chi \leq  \kappa \xi - (\kappa -1)$. This is the most involved part of the proof of Theorem \ref{theorem1}. In Section 3, we prove the lower bound by the same argument initially given by Newman - Piza \cite{NP}. In Appendix A we prove Proposition 1 and the Shape Theorem while in Appendix B we establish a lemma that estimates the rate of convergence of $F(\bb 0, \bb x)$ towards $f(\bb x)$.

\section{Proof of $\chi \leq \kappa \xi - (\kappa -1)$}

To prove the upper bound $\chi \leq  \kappa \xi - (\kappa -1)$ we will follow the strategy of \cite{AD2}.
We start with a lemma. Write $I(A)$ for the indicator function of the event $A$.

\begin{lem}\label{lem: rvs}
Let $X$ and $Y$ be random variables with $\|X\|_4,\|Y\|_4<\infty$ and let $B$ be an event such that
for some $\e >0$,
\[
|X-Y|I(B) \leq \e \text{ almost surely.}
\]
Then
\begin{equation}\label{eq: rvs}
|\Var~X -\Var~Y| \leq   \|X-Y\|_4(\|X\|_2+\|Y\|_2) \mathbb{P}(B^c)^{1/4} + \e(\|X\|_2+\|Y\|_2) \ .
\end{equation}
\end{lem}

\begin{proof}
Let $\widetilde X = X - \mathbb{E} X$ and $\widetilde Y = Y-\mathbb{E} Y$. The left side of \eqref{eq: rvs} equals
\begin{eqnarray*}
\left|\|\widetilde X\|_2^2 - \|\widetilde Y\|_2^2\right| &=& \left|\|\widetilde X\|_2 - \|\widetilde Y\|_2\right| ~\left|\|\widetilde X\|_2 + \|\widetilde Y\|_2\right| \\
&\leq& \|X-Y\|_2 (\|X\|_2 + \|Y\|_2)  \\
&=&\|(X-Y)I(B^c) + (X-Y)I(B)\|_2 (\|X\|_2 + \|Y\|_2) \\
&\leq& \|X-Y\|_4(\|X\|_2+\|Y\|_2) \mathbb{P}(B^c)^{1/4} + \e(\|X\|_2+\|Y\|_2) \ .
\end{eqnarray*}
\end{proof}

Note that by \cite[Proposition~1(b)]{Piza}, $\chi_b \leq 1/2$.   Therefore if $\xi_a = 1$ then the bound $\chi \leq \kappa \xi - (\kappa -1)$ holds.
Because we will deal with the case $\chi = 0$ in a later argument, we will now assume that
\begin{equation}\label{eq: assumption2}
\xi_a<1 \text{ and } \chi_b >0
\end{equation}
so that we can choose $\xi'$ and $\chi''$ such that
\begin{equation}\label{eq: xichiprime}
\xi_a<\xi'<1\text{ and } 0<\chi''<\chi_b\ .
\end{equation}

Let $\bb v_n$ be a point in $\Z^d$ with $ \bb v_n \cdot \bb e  = 0$  and $|\bb v_n| \in [2n^{\xi'}, 3n^{\xi'}].$ Set
\[
\delta F(n,\xi') = F(\bb 0, n \bb e) - F(\bb v_n, \bb v_n + n \bb e)\ .
\]

\subsection{Lower bound on $\Var \delta F(n,\xi')$ }
\begin{prop}\label{propdown}
Assume \eqref{eq: assumption2}. For each $\xi'$ and $\chi''$ chosen as in \eqref{eq: xichiprime}, there exists $C=C(\xi',\chi'')$ such that for all $n$,
\[
\Var~ \delta F(n,\xi') \geq Cn^{2\chi''}\ .
\]
\end{prop}

\begin{proof}
Let $\mathcal{C}_1 = C_{n \bb e}[n^{\xi'}]$ and $\mathcal{C}_2 = \mathcal{C}_1 + \bb v_n$.  Note that by our choice of $\bb v_n$, $\mathcal{C}_1 \cap \mathcal{C}_2 = \varnothing$. We now define the restricted partition functions $Z_1(n)'$ and $Z_2(n)'$ as follows:

$$Z_1(n)' = \sum_{\gamma: \bb 0 \rightarrow n\bb e, \gamma \subseteq \mathcal{C}_1} \exp(-\beta \tau(\gamma)), \quad Z_2(n)'=  \sum_{\gamma: \bb v_n \rightarrow \bb v_n + n\bb e, \gamma \subseteq \mathcal{C}_2} \exp(-\beta \tau(\gamma)) $$
with the corresponding free energies $F_1'$ and $F_2'$ as in \eqref{eq:freeenergy}.

Note that $F_1'$ and $F_2'$ are independent random variables with the same distribution. We will now show that given our choice of the size of the cylinder $\mathcal C_1$, the variance of $F(\bb 0, n \bb e)$ cannot be much higher than the variance of  $F_1'$.

Let $\alpha = \alpha(\xi')$ be given as in Definition~\ref{def2}. Let $B$ be the event $\{F (\bb 0, n \bb e) \geq F_1' - \frac{1}{\beta}n^{-(1+\alpha)} \text{ and } F(\bb v_n, \bb v_n + n \bb e)\geq F_2' - \frac{1}{\beta} n^{-(1+\alpha)}\}$. Note that $$B^c \subseteq \{ \log \mu_{\bb 0, n \bb e} (\mathcal C_1) \leq - n^{-(1+\alpha)} \}  \cup \{ \log \mu_{\bb v_n, \bb v_n + n \bb e}(\mathcal C_2) \leq - n^{-(1+\alpha)} \}\ .$$ Therefore from the inequality $\exp (-x) \leq 1 - \frac{1}{2}x$ for $x$ small and positive and by the definition of $\xi_a$ there exists $\delta >0$ so that $\mathbb{P}(B^c) \leq 2e^{-\delta n}$ for $n$ large enough.

By Lemma~\ref{lem: rvs} with $X=\delta F(n,\xi') $,  $Y=\delta F(n,\xi')':= F_1' - F_2'$ and $\e = \frac{2}{\beta}n^{-(1+\alpha)}$ there exists $C_1 >0$ such that
\begin{equation*}
\begin{split}
\Var~ &\delta  F(n,\xi')   \\  &\geq \Var~\delta F(n,\xi') ' - (\|\delta  F(n,\xi') \|_2 + \| \delta F(n,\xi') '\|_2) (\e 
+ \mathbb{P}(B^c)^{1/4} \| \delta  F(n,\xi') - \delta  F(n,\xi')' \|_4)\\
&\geq\Var~\delta  F(n,\xi') ' - C_1n^2 e^{-\delta n/4} -   C_1n^{-\alpha}\ .
\end{split}
\end{equation*}
Here we have used that each $\delta F$ is a difference of logarithms of partition functions, each of which has $L^4$ norm bounded above by $C n$ (compare for example to the contribution given by a deterministic path) for some constant $C$. Therefore there exists a constant $C_2$ such that for all $n$,
\begin{equation}\label{eq: last1}
\Var~\delta  F(n,\xi') \geq\Var~\delta  F(n,\xi')'  - C_2\ .
\end{equation}
But $\delta  F(n,\xi')' $ is the difference of i.i.d. random variables distributed as $F_1'$, so 
\begin{equation}\label{eq: last2}
\Var~\delta  F(n,\xi')'  = 2 \Var~ F_1(n)'\ .
\end{equation}
By exactly the same argument as that given above, we can find $C_3$ such that for all $n$,
\[
\Var~ F_1(n)'\geq \Var F (\bb 0, n \bb e) - C_3\ .
\]
Now, combining \eqref{eq: last1} with \eqref{eq: last2} and using the definition of $\chi''$, we can find $C_4$ such that for all $n$, $\Var~ \delta F(n,\xi') \geq C_4n^{2\chi''}$. 
\end{proof}

\subsection{Upper bound on $\Var \delta F(n,\xi')$}\label{subsec: upperbound}

In this section  we work with the same choice of $\xi'$ that satisfies \eqref{eq: xichiprime}. We will prove the following.
\begin{prop}\label{propup}
Assume \eqref{eq: assumption2} and that \eqref{def1} holds for some $C_1,C_2,\e$ and $\kappa$. For each $\eta$ satisfying $\xi'<\eta<1$ and each $\chi'>\chi_a$, there exists $D=D(\eta,\chi')$ such that for all $n$,
\[
\Var~\delta F(n,\xi') \leq Dn^{2\eta(1-\kappa)+2\xi'\kappa} + Dn^{2\eta \chi'}\ .
\]
\end{prop}

\begin{proof}

Let $\mathcal{C}_1$ and $\mathcal{C}_2$ be as in the proof of the lower bound. Let $\tilde B$ be the convex hull of $ \mathcal{C}_1 \cup \mathcal{C}_2$. Define

\[ L_1 = \{ \bb v \in \tilde B: \bb v \cdot \bb e = 0 \}, \;  R_1 = L_1 + \lfloor n^{\eta}  \rfloor \bb e \]
and $L_2 = L_1 + (n- \lfloor n^{\eta} \rfloor) \bb e$, $R_2 = L_1 + n\bb e$.  Let $\tilde Z (\bb u, \bb v)$ be the constrained partition function  from $\bb u$ to $\bb v$ only considering paths that intersect both $R_1$ and $L_2$ and define the corresponding free energy $\tilde F(\bb u, \bb v)$. Set $$ \tilde F_1 =   \tilde F(\bb 0, n \bb e), \quad \tilde F_2 = \tilde F (\bb v_n, \bb v_n + n \bb e).$$

As  in the last section, if $B$ is the event $\{F (\bb 0, n \bb e) \geq \tilde F_1 - \frac{1}{\beta}n^{-(1+\alpha)} \text{ and } F(\bb v_n, \bb v_n + n \bb e)\geq \tilde F_2 - \frac{1}{\beta} n^{-(1+\alpha)}\}$,  Lemma \ref{lem: rvs} implies that there exists a constant $C_5$ such that 
\begin{equation}\label{eq: happy_time}
\Var~\delta  F(n,\xi') \leq\Var~ (\tilde F_1 - \tilde F_2)  + C_5\ . 
\end{equation}
Therefore it suffices to bound $\Var~ (\tilde F_1 - \tilde F_2)$, which is equal to $ \|\tilde F_1 - \tilde F_2 \|_2^2$.

To do this, let
\begin{equation}
M_i = \max_{\bb u \in L_i, \bb v  \in R_i} Z (\bb u, \bb v), \; m_i = \min_{\bb u \in L_i, \bb v  \in R_i} Z (\bb u, \bb v) \text{ for } i=1,2\ .
\end{equation}

Now,
\begin{equation}\label{eq:27}
\begin{split}
|\tilde F_1- \tilde F_2| = \bigg|-\frac{1}{\beta} \log \frac{ \tilde Z(\bb 0, n \bb e) }{\tilde Z(\bb v_n, \bb v_n + n \bb e)}\bigg| &= \frac{1}{\beta}\bigg| \log \frac{ \sum_{\bb y\in R_1, \bb y' \in L_2 } Z(\bb 0, \bb y) Z(\bb y, \bb y') Z(\bb y', n \bb e)}{\sum_{\bb y\in R_1, \bb y' \in L_2 } Z(\bb v_n, \bb y) Z(\bb y, \bb y') Z(\bb y', \bb v_n + n \bb e)}\bigg| \\
&\leq \frac{1}{\beta} \bigg|\log \frac{M_1 M_2}{m_1 m_2} \bigg|.
\end{split}
\end{equation}

\begin{lemma} \label{lem:lemma} There exists a constant $C_6$ such that for all $n$
\[ 
\E |\log M_1 - \log m_1|^2 \leq C_6 n^{2\eta \chi'} + C_6 n^{2(\eta - \kappa(\eta - \xi'))} \ .
\]
\end{lemma}

\begin{proof} Note that

\begin{equation}
\begin{split}
\E |\log M_1 - \log m_1|^2 &\leq \E\bigg( \max_{\stackrel{\bb u_1 \in L_1, \bb v_1 \in R_1}{  \bb u_2 \in L_1,\bb v_2 \in R_1}} |\log Z(\bb u_1, \bb v_1) - \log Z(\bb u_2, \bb v_2) | ^2\bigg) \\
&\leq 4 \E\bigg( \max_{\bb u_1 \in L_1,\bb v_1 \in R_1} |\log Z(\bb 0, n^{\eta}\bb e) - \log Z(\bb u_1, \bb v_1) | ^2\bigg).
\end{split}
\end{equation}

Now \begin{equation}
\max_{\bb u_1 \in L_1, \bb v_1 \in R_1} |\log Z(\bb 0, n^{\eta}\bb e) - \log Z(\bb u_1, \bb v_1) | \leq I + II
\end{equation}
where 
\begin{eqnarray*}
I &=& |\log Z(\bb 0, n^{\eta}\bb e) + \beta n^{\eta}f(\bb e)| +  \max_{\bb u_1 \in L_1, \bb v_1 \in R_1} | \log Z(\bb u_1, \bb v_1) + \beta f(\bb v_1- \bb u_1)|\  ,\\
II &=& \beta \max_{\bb u_1 \in L_1, \bb v_1 \in R_1} |f(\bb u_1 - \bb v_1)- f(n^{\eta}\bb e)| \ .
\end{eqnarray*}
To estimate the second term, note that for any $\bb u_1 \in L_1$ and $\bb v_1 \in R_1$
\begin{equation}\label{jdasd}\begin{split}
|f(\bb v_1 - \bb u_1)- f(n^{\eta}\bb e)| = n^{\eta} \left| f\left(\frac{\bb v_1- \bb u_1}{n^{\eta}} - \bb e + \bb e \right) - f(\bb e) \right| &\leq C_2 n^{\eta} \left|  \frac{\bb v_1 - \bb u_1}{n^{\eta}} - \bb e \right|^{\kappa} \\
&\leq C_7 n^{\eta - \kappa(\eta - \xi')}\ ,
\end{split}
\end{equation}
where we used the curvature assumption \eqref{eqassumption} and the fact that $\eta > \xi'$.

The estimation of $I$ follows directly from Lemma \ref{lem: AC}. Indeed, taking $ \chi_a < \hat \chi < \chi'$, it provides $\alpha>0$ such that 
\begin{equation}\label{eq:oopa}
 \sup_{\bb u_1 \in L_1, \bb v_1 \in R_1} \E \exp \bigg(\alpha  \frac{|\log Z(\bb u_1, \bb v_1) + \beta f(\bb u_1 - \bb v_1)|}{|\bb u_1 - \bb v_1|^{\hat \chi}}\bigg) < \infty.
 \end{equation}

Now note that for any $\alpha>0$ and any positive random variable $X$ one has

\begin{equation}\label{simpleineq}
\|X\|_2 \leq \frac{1}{\alpha} \log 2\E e^{\alpha X} \ .
\end{equation} 
This can be seen by Jensen's inequality as
\begin{equation}\label{eq: pizza}
e^{\alpha \|X\|_2} = 1 + \alpha \|X\|_2 + \sum_{n=2}^\infty \frac{(\alpha \|X\|_2)^n}{n!} \leq 1 + \alpha \|X\|_2 + \E \sum_{n=2}^\infty \frac{(\alpha X)^n}{n!} \leq \alpha \|X\|_2 + \E e^{\alpha X}\ .
\end{equation}
Because $e^{\alpha \|X\|_2} \geq 2 \alpha \|X\|_2$, we must have $\alpha \|X\|_2 \leq \E e^{\alpha X}$, so $e^{\alpha \|X\|_2} \leq 2 \mathbb{E} e^{\alpha X}$. Taking logarithms, we find \eqref{simpleineq}.

Applying \eqref{simpleineq} to $$X = \max_{\bb u_1 \in L_1, \bb v_1 \in R_1}   \frac{| \log Z(\bb u_1, \bb v_1) + \beta f(\bb v_1- \bb u_1)|}{|\bb v_1 - \bb u_1|^{\hat \chi}}$$
and using \eqref{eq:oopa} we obtain an upper bound for $\mathbb{E} X^2$ of
$$ \left(\frac{1}{\alpha} \log 2 \E e^{\alpha X}\right)^2 \leq \left(\frac{1}{\alpha} \log 2 \sum_{\bb u_1, \bb  v_1} \exp \left[ \alpha \E \frac{| \log Z(\bb u_1, \bb v_1) + \beta f(\bb v_1- \bb u_1)|}{|\bb v_1 - \bb u_1|^{\hat \chi}}\right]\right)^2 \leq C_8 (\log n)^2.$$
Since $|\bb v_1 - \bb u_1|^{\hat \chi} \leq C_9 n^{\eta \hat \chi}$ this immediately implies that 
\begin{equation} \label{eq:eq23}
\E I^2 \leq C_{10} n^{2\eta \chi'}.
\end{equation}
Hence, combining \eqref{eq:eq23} and \eqref{jdasd} we finish the proof of the lemma. 

\end{proof}

Going back to the proof of the Proposition, using Lemma 1 and \eqref{eq:27}  we see that since $\tilde F_1$ and $\tilde F_2$ have the same distribution 
$$ \Var (\tilde  F_1 - \tilde F_2) \leq  \frac{4}{\beta^2} (C_6 n^{2\eta \chi'} + C_6 n^{2(\eta - \kappa(\eta - \xi'))})\ .$$
Using \eqref{eq: happy_time}, this ends the proof of Proposition \ref{propup}.

\end{proof}

\subsection{Proof of $\chi \leq \kappa \xi - (\kappa -1)$}

In this section we prove one of the two inequalities for the relation \eqref{eqAD}. We first show that $\chi \geq 0$. We then split the proof into two cases depending on the value of $\chi$. The proof for $\chi>0$ will follow from the previous sections and the proof for $\chi=0$ will be essentially a rewrite of Chatterjee \cite[Section~9]{Sourav}.

\subsubsection{$\chi$ is always non-negative}

We follow the analogous proof of Chatterjee \cite[Section 3]{Sourav}. To prove that $\chi \geq 0$ it suffices to show the existence of a constant $C>0$ such that for any $\bb v \in \Z^d_{+} \setminus \{\bb 0\}$ , $\Var F(\bb 0, \bb v) \geq C.$ We proceed as follows. Assume that the edge-weights are non-degenerate. Let $E$ be the collection of edges incident to the origin. Let $c_1 < c_2$ be positive constants such that 
\begin{equation*}
\Pro(\max_{e \in E} \tau_e \leq c_1) >0 \; \text{and} \; \Pro( \min_{e \in E} \tau_e \geq  c_2 ) > 0\ .  
\end{equation*}
Define a new environment $\tau_e'$ such that $\tau_e' = \tau_e$ if $e \notin E$ and $\tau_e'$ is a independent copy of $\tau_e$ if $e \in E$. Let $F'$ be the corresponding free energy for the environment $\tau'$ and $\mathcal F$ be the sigma-algebra generated by the edges $e \notin E$. Under the event  $\max_{e \in E} \tau_e \leq c_1$ and  $\min_{e \in E} \tau_e' \geq  c_2$ one has that for all $\bb v \in \Z^d_{+} \setminus \{\bb 0\}$, $|F(\bb 0, \bb v)-F'(\bb 0, \bb v)| > c_2 - c_1>0$. Therefore $$\E \Var (F(\bb 0, \bb v) \mid  \mathcal F) = \frac{1}{2} \E \left[ \E ( |F(\bb 0, \bb v)-F'(\bb 0, \bb v)|^2 \mid \mathcal{F} )\right] > \frac{1}{2} (c_2 - c_1)^2 > 0$$
which implies that for any $\bb v \in \Z^d_{+} \setminus \{\bb 0\}$, $\Var F(\bb 0, \bb v) \geq C$ with $C = \frac{1}{2} (c_2 - c_1)^2$. \subsubsection{The case $\chi>0$}

We combine Propositions \ref{propdown} and \ref{propup}. Indeed, it follows from these propositions that for any $\eta$ satisfying $\xi' < \eta < 1$ and any $\chi'' < \chi < \chi'$ one has positive constants $C_1$, $C_2$ such that for all $n \geq 1$,

$$ C_1 n^{2 \chi''} \leq  C_2n^{2\eta(1-\kappa)+2\xi'\kappa} + C_2n^{2\eta \chi'}. $$

For any $\eta$ with $\xi'<\eta<1$, we may choose $\chi'' = \chi''(\eta)$ and $\chi'= \chi'(\eta)$ (both converging to $\chi$ as $\eta \to 1$) that are so close to $\chi$ that $2\eta\chi' < 2\chi''$. This implies that for all  $n$ large enough $\frac{C_1}{2} n^{2 \chi''} \leq  C_2 n^{2\eta(1-\kappa)+2\xi' \kappa}$.  This can only hold if $ \chi'' \leq \eta(1-\kappa)+\xi'\kappa $. Taking $\eta$ to $1$ and therefore $\chi''$ to $\chi$ we obtain 
$$ \chi \leq  \kappa \xi - (\kappa -1)\ .$$

\subsubsection{The case $\chi=0$}
In this section we prove the inequality $\chi \leq \kappa \xi - (\kappa-1)$ in the case $\chi=0$, beginning with a lemma that replaces \cite[Lemma~9.1]{Sourav}. For $M>0$, let $F^{(M)}(\bb 0, \bb x)$ be the free energy of all paths from $\bb 0$ to $\bb x$ in the constant environment, where each edge-weight equals $M$.

\begin{lemma}\label{lem: souravslemma}
Assume that $\mathbb{P}(\tau_e=L)<1$, where $L$ is the infimum of the support of the distribution of $\tau_e$ and $\mathbb{E} \tau_e^{d+\alpha}<\infty$ for some $\alpha>0$. There exists $M > L$ such that
\[
\mathbb{P}\left( F(\bb 0, \bb x) \geq F^{(M)}(\bb 0, \bb x) \text{ for all but finitely many } \bb x \in \mathbb{Z}_+^d \right) = 1\ .
\]
\end{lemma}
\begin{proof}
Because of the shape theorem and Lemma~\ref{lem:lem2ls}, it suffices to show that for some $M>L$,
\[
\mathbb{E} F(\bb 0, \bb x) \geq F^{(M)}(\bb 0, \bb x)
\]
for all nonzero $\bb x \in \mathbb{Z}_+^d$. We do this by a computation similar to that given in the proof of Proposition~\ref{prop:pro1}, item 2. Write $N(\bb 0, \bb x)$ for the number of directed paths from $\bb 0$ to $\bb x$. We first consider the case $L=0$ and use Jensen's inequality:
\[
\mathbb{E} F(\bb 0, \bb x)  \geq -\frac{1}{\beta} \log \frac{\sum_{\gamma : \bb 0 \to \bb x} \mathbb{E} e^{-\beta \tau(\gamma)}}{d^{|\bb x|_1}} = \frac{1}{\beta}(|\bb x|_1 \log d - \log N(\bb 0, \bb x)) - \frac{|\bb x|_1}{\beta} \log \mathbb{E} e^{-\beta \tau_e}\ .
\]
On the other hand,
\[
F^{(M)}(\bb 0, \bb x) = -\frac{1}{\beta} \log \frac{e^{-\beta M|\bb x|_1} N(\bb 0, \bb x)}{d^{|\bb x|_1}} = \frac{1}{\beta}(|\bb x|_1 \log d -\log N(\bb 0, \bb x)) + M|\bb x|_1\ .
\]
So choosing $M < -\frac{1}{\beta} \log \mathbb{E} e^{-\beta \tau_e}$ (which is positive by assumption), the proof is complete.

In the case $L>0$ we define new edge-weights $(s_e)$ by $s_e = \tau_e-L$. Define $F^s(\bb 0, \bb x)$ in the same way as $F(\bb 0, \bb x)$ but for the weights $(s_e)$. By the above argument, we find $K>0$ such that 
\[
\mathbb{P}\left(F^s(\bb 0, \bb x) \geq F^{(K)}(\bb 0, \bb x) \text{ for all but finitely many } \bb x \in \mathbb{Z}_+^d\right)=1\ .
\]
But $F^s(\bb 0, \bb x) + L|\bb x|_1 = F(\bb 0, \bb x)$, so we can set $M=K+L$.
\end{proof}


\begin{proof}[Proof of $\chi \leq \kappa \xi - (\kappa-1)$ in the case $\chi=0$.] In the rest of this section, we essentially copy \cite{Sourav} with minor changes. We will prove the inequality by contradiction. Assume that $\chi=0$ and $\kappa \xi - (\kappa-1) < \chi$. Then $\xi < (\kappa-1)/\kappa$. Choose $\xi'$ such that
\[
\xi <\xi'<(\kappa-1)/\kappa\ .
\]
Let $\delta = \delta(\xi')$ be as in the definition of $\xi_a$.

Choose $\zeta, r'$ and $r$ such that $0<r'<r<\zeta<\delta/d$ and $\zeta < \xi'$. Let $n$ be a positive integer, to be chosen large at the end of the proof. Choose any $\bb z$ with $\bb z \cdot \bb e = 0$ and $|\bb z|_1 \in (n^{\xi'},2n^{\xi'}]$. Let $\bb w = n\bb e/2 +\bb z$. Then because $\xi' < (\kappa-1)/\kappa$, there exists $C_1$ such that for all $n$,
\[
|f(\bb w) - f(n\bb e/2)| \leq C_1\ .
\]
Similarly,
\[
|f(n\bb e-\bb w) - f(n\bb e/2)| \leq C_1\ .
\]
Therefore, for all $n$,
\begin{equation}\label{eq: pizza1}
|f(n\bb e) - (f(\bb w) - f(n\bb e-\bb w))| \leq C_2\ .
\end{equation}

By Lemma~\ref{lem: AC} and the assumption that $\chi=0$, the probabilities $\mathbb{P}(|F(\bb 0,\bb w) - f(\bb w)| > n^r)$, $\mathbb{P}(|F(\bb w,n\bb e) - f(n\bb e-\bb w)| > n^r)$ and $\mathbb{P}(|F(\bb 0,n\bb e) - f(n\bb e)| > n^r)$ are all bounded by $e^{-C_3n^{r-r'}}$ for some $C_3$ depending on $r$ only. These observations, along with \eqref{eq: pizza1}, imply that there are constants $C_4$ and $C_5$, independent of our choice of $n$ such that
\begin{equation}\label{eq: pizza2}
\mathbb{P}(|F(\bb 0,n\bb e) - (F(\bb 0,\bb w) + F(\bb w,n\bb e))| > C_4 n^r) \leq e^{-C_5 n^{r-r'}}\ .
\end{equation}
By the definition of $\xi_a$, there exists $C_6$ such that
\begin{equation}\label{eq: tocontradict}
\mathbb{P}( \mu(\gamma \in C_{n\bb e}[n^{\xi'}]) > 1-e^{-\beta n^r}) \geq 1-C_6\exp(-n^\delta)\ .
\end{equation}
Let $F_0(\bb 0,n\bb e)$ be the free energy of all paths from $\bb 0$ to $n\bb e$ that stay inside of the cylinder $C_{n\bb e}[n^{\xi'}]$. Inequality \eqref{eq: tocontradict} means in particular that
\[
\mathbb{P}(F_0(\bb 0,n\bb e) - F(\bb 0,n\bb e) \leq n^r) \geq 1-C_6\exp(-n^\delta)\ .
\]

Combining this with \eqref{eq: pizza2}, we see that if $E_1$ is the event
\[
E_1 := \{|F_0(\bb 0,n\bb e) - (F(\bb 0,\bb w) + F(\bb w,n\bb e))| \leq C_7n^r\}\ ,
\]
(for $C_7 = C_4+1$) then
\begin{equation}\label{eq: pizza3}
\mathbb{P}(E_1) \geq 1-C_6e^{-n^\delta} - e^{-C_5 n^{r-r'}}\ .
\end{equation}

Let $V$ be the set of all lattice points within $\ell_1$ distance $n^\zeta$ from $\bb w$. Let $\partial V$ be the set of $\bb v \in V$ which have one neighbor outside of $V$. Write $\partial_1 V$ for the set of points $\bb v \in \partial V$ with $\bb v \leq \bb w$. Letting $L,M$ be as in Lemma~\ref{lem: souravslemma}, we have
\[
\mathbb{P}(E_2) \to 1 \text{ as } n \to \infty\ ,
\]
where $E_2$ is the event that $F(\bb v,\bb w) \geq F_M(\bb v, \bb w)$ for all $\bb v \in \partial_1 V$.

Let $E(V)$ denote the set of edges in directed paths from vertices in $\partial_1 V$ to $\bb w$. Let $(\tau_e')_{e \in E(V)}$ be a collection of i.i.d. random variables, independent of the original edge-weights, but having the same distribution. For $e \notin E(V)$ let $\tau_e' = \tau_e$. Choosing $L'$ such that $L<L'<M$, let $E_3$ be the event
\[
E_3 := \{\tau_e' \leq L' \text{ for all } e \in E(V)\}\ .
\]
If $E_3$ occurs, then for each directed path $\sigma$ from a vertex in $\partial_1 V$ to $\bb w$, $\tau'(\sigma) \leq L'n^\zeta$ and therefore $F'(\bb v, \bb w) \leq F^{(L')}(\bb v, \bb w)$, where $F'(\bb v, \bb w)$ is defined the same way as $F(\bb v, \bb w)$ but for the weights $(\tau_e')$. We can estimate
\begin{align*}
F(\bb 0, \bb w) - F'(\bb 0, \bb w) &= -\frac{1}{\beta}\log \frac{\sum_{\bb v \in \partial_1 V} e^{-\beta F(\bb 0, \bb v)}  e^{-\beta F(\bb v, \bb w)}}{\sum_{\bb v' \in \partial_1 V} e^{-\beta F'(\bb 0, \bb v')} e^{-\beta F'(\bb v', \bb w)}} \\
&= -\frac{1}{\beta} \log \frac{\sum_{\bb v \in \partial_1 V} e^{-\beta F(\bb 0, \bb v)} e^{-\beta(F(\bb v, \bb w)-F'(\bb v, \bb w))} e^{-\beta F'(\bb v, \bb w)}}{\sum_{\bb v' \in \partial_1 V} e^{-\beta F(\bb 0, \bb v')}e^{-\beta F'(\bb v', \bb w)}}\ .
\end{align*}
On the event $E_2 \cap E_3$, we have 
\[
F(\bb v, \bb w) - F'(\bb v, \bb w) \geq F(\bb v, \bb w) - F^{(M)}(\bb v, \bb w) + F^{(M)}(\bb v, \bb w) - F^{(L')}(\bb v, \bb w) \geq (M-L') n^\zeta
\]
and therefore $F(\bb 0, \bb w) - F'(\bb 0, \bb w) \geq (M-L')n^\zeta$. This means that if all of the events $E_i,$ $i=1,2,3$ occur simultaneously then
\begin{eqnarray*}
F_0(\bb 0,n\bb e) &\geq& F(\bb 0,\bb w) + F(\bb w,n\bb e) - C_7n^r \\
&\geq& F'(\bb 0,\bb w) + F'(\bb w,n\bb e) -C_7n^r + (M-L')n^\zeta\ .
\end{eqnarray*}
As $\zeta>r$, we would then have, for some $C_8$,
\[
\mu_{\bb 0,n\bb e}'(\gamma \in C_{n\bb e}[n^{\xi'}]) \leq e^{-C_8 n^\zeta}\ ,
\]
where $\mu_{\bb 0,n\bb e}'$ is the Gibbs measure for the weights $(\tau_e')$. 

Since the intersection $\cap_{i=1}^3 E_i$'s occurs with probability at least $e^{-C_9 n^{\zeta d}}$, 
\[
\mathbb{P}(\mu_{0,n\bb e}(\gamma \in C_{n\bb e}[n^{\xi'}]) \geq e^{-C_8 n^\zeta}) \leq 1-e^{-C_9 n^{\zeta d}}\ .
\]
Recalling that $\zeta d < \delta$, this contradicts \eqref{eq: tocontradict}.

\end{proof}

\section{Proof of the lower bound $\chi \geq \kappa \xi - (\kappa -1)$}

The argument below was initially given for zero temperature in the work of Newman - Piza \cite{NP}  as a rigorous version of one by Krug - Spohn and for positive temperature (but with a different definition of exponents than the ones we consider here) by Piza \cite{Piza}. It was adapted by others, including Chatterjee \cite{Sourav}, in several different models. Since this argument has appeared so many times in the literature we try to be brief in this section and leave some details to the reader.

The proof  will proceed by contradiction. Suppose that $\chi < \kappa \xi - (\kappa -1).$ Choose $\xi'$
such that
$$\frac{\chi + \kappa -1}{\kappa}  < \xi' <\xi \leq 1 \ . $$

Let $V$ be the set of all lattice points $\bb v$ in the set $C_{n\bb e} [2 n^{\xi'}] \setminus C_{n\bb e} [ n^{\xi'}] 
$ such that $\bb 0 \leq \bb v \leq n \bb e$. We first claim that there is a constant $C_1$ such that for any $\bb v \in V$ and any $n \in \N$,
\begin{equation}\label{curvaturel}
f(\bb v) + f(n\bb e- \bb v) \geq f(n \bb e) + C_1n^{\kappa \xi' - (\kappa -1)}\ .
\end{equation}
Indeed, by symmetry, we may assume that $\bb v$ has Euclidean norm at least $\frac{n}{2}$. Let $\bb w$ be the orthogonal projection of $\bb v$ onto $ \bb e$.   By convexity of $f$ we have $$f(\bb v) + f(n\bb e- \bb v) - f(n \bb e) =  f(\bb v) - f(\bb w)+ f(n\bb e- \bb v) - f(n \bb e - \bb w) \geq f(\bb v) - f(\bb w)\ ,$$ 
but also $$ f(\bb v) - f(\bb w) =  f(\bb v - \bb w + \bb w) - f(\bb w) \geq C_1n^{\kappa \xi' - (\kappa -1)}$$ by Assumption \ref{assumption1}.

Now,  take $\chi_1, \chi_2$ such that $\chi < \chi_1 <  \chi_2 < \kappa \xi' - (\kappa -1)$. Then by Lemma \ref{lem: AC}, there is a constant $C_2$ such that for $n$ large enough, the following three inequalities hold: 
\begin{eqnarray*}
\Pro\bigg( F(\bb 0, n \bb e) > nf(\bb e) + n^{\chi_2}\bigg)  &\leq& \exp \big( -C_2n^{\chi_2 - \chi_1}\big) \ , \\
\Pro \bigg(F(\bb 0,  \bb v) < f( \bb v) - n^{\chi_2}\bigg) &\leq& \exp \big( -C
_2n^{\chi_2 - \chi_1}\big) \ , \\
\Pro \bigg( F(\bb v, n \bb e) < f( n \bb e - \bb v) - n^{\chi_2}\bigg) &\leq&\exp \big( -C_2n^{\chi_2 - \chi_1}\big)\ .
\end{eqnarray*}

This combined with $\kappa \xi' - (\kappa-1) > \chi_2$  shows that for some $C_3 > 0$ if $n$ is large enough, for any $\bb v \in V$, , $$\Pro \bigg(F(\bb 0,  n \bb e) \geq F(\bb 0,  \bb v)  + F(\bb v,  n \bb e) - C_3 n^{\kappa \xi' - (\kappa-1)} \bigg) \leq 3\exp \big( -C_2n^{\chi_2 - \chi_1}\big)\ . $$
The size of $V$ is a polynomial function in $n$. This implies that there exists $C_4>0$ such that $$\Pro\bigg(F(\bb 0,  n \bb e) \geq F(\bb 0,  \bb v)  + F(\bb v,  n \bb e) - C_3n^{\kappa \xi' - (\kappa -1)}   \; \text{for some} \; \bb v \in V \bigg) \leq  \exp \big( -C_4n^{\chi_2 - \chi_1}\big)\ . $$ 
Note that this translates to $$ \Pro \bigg( \mu_{\bb 0, n\bb e} (\{\gamma : \bb v \in \gamma \}) \leq e^{-  \beta C_3 n^{\kappa \xi' - (\kappa-1)} }  \; \text{for some} \; \bb v \in V \bigg) \leq \exp(-C_4n^{\chi_2-\chi_1})\ ,$$
and therefore for some $C_5 >0$ we have
 $$ \Pro \bigg( \mu_{\bb 0, n\bb e} (\{\gamma : \bb v \in \gamma \; \text{for some} \; \bb v \in V \}) \leq e^{-  \beta C_5 n^{\kappa \xi' - (\kappa-1)} }  \;  \bigg) \leq \exp(-C_4n^{\chi_2-\chi_1})\ .$$
Now, an application of Borel-Cantelli shows that $\xi'$ is such that for all $\e >0$ $$ \Pro \bigg( \mu_{\bb 0, n\bb e} (\gamma \in C_{n\bb e}[n^{\xi'}] )> 1- \e \bigg)\rightarrow 1 $$ and this contradicts the definition of $\xi_b$.

\appendix

\section{Proof of Proposition~\ref{prop:pro1} and the Shape Theorem} \label{appendix}
In this section, we prove Propositions~\ref{prop:pro1} and~\ref{shapethm}. We start with a concentration lemma that will be used in both propositions. Let $\bb z_1, \ldots, \bb z_k, \bb z \in \mathbb{Z}^d_+$ such that
\[
\bb z_1 \leq \cdots \leq \bb z_k \leq \bb z \ .
\]
Define the free energy of all paths that pass through all $\bb z_i$'s from $\bb 0$ to $\bb z$ as
\[
F(\bb 0, \bb z; \vec{\bb z}) = F(\bb 0, \bb z_1, \ldots, \bb z_k, \bb z)\ .
\]

\begin{lemma}\label{lem:conc} Let $\vec{\bb z} = (\bb z_1, \ldots, \bb z_k)$ and $\bb z$ be as above. Assume that $\mathbb{P}(\tau_e\leq L) = 1$. For any $t >0$,
\begin{equation*}
\Pro \bigg( | F(\bb 0, \bb z; \vec{\bb z}) - \E F(\bb 0, \bb z; \vec{\bb z}) | > t \sqrt{| \bb z |_1}\bigg) \leq 2 \exp \bigg( - \frac{t^2}{2L^2}\bigg)\ .
\end{equation*} 
\end{lemma}

\begin{proof} Let $\mathcal F_0$ denote the trivial sigma-algebra and $\mathcal F_j,~j \geq 1$ be the sigma-algebra generated by the weights $\tau_e$ such that both endpoints of $e$ have $\ell^1$ norm no bigger than $j$. To prove the lemma, we will write $F (\bb 0, \bb z; \vec{\bb z}) - \E F(\bb 0, \bb z; \vec{\bb z})$ as a sum of $| \bb z |_1$ martingale differences:
\[
F(\bb 0, \bb z; \vec{\bb z}) - \E F(\bb 0, \bb z; \vec{\bb z}) =  \sum_{j=1}^{| \bb z |_1} D_{j} -D_{j-1},~ \text{where } D_j = \E \big( F(\bb 0, \bb z; \vec{\bb z}) \mid  \mathcal F_j \big)\ .
\]
For a fixed $j$, write $F[\tau^{(1)},\tau^{(2)},\tau^{(3)}]$ for $F(\bb 0, \bb z; \vec{\bb z})$ as a function of the edge weights for edges with both endpoints of $\ell^1$-norm no bigger than $j$ ($\tau^{(1)}$), strictly bigger than $j$ ($\tau^{(3)}$) and all other edges ($\tau^{(2)}$). The bound on the edge weights implies that if $(\tau_e)$ and $(\tilde \tau_e^{(2)})$ are sampled independently from $\mathbb{P}$ then
\[
|F[\tau^{(1)},\tau^{(2)},\tau^{(3)}] - F[\tau^{(1)},\tilde \tau^{(2)}, \tau^{(3)}] | \leq L \quad \mathbb{P}\text{-almost surely}\ .
\]
Therefore a calculation gives
\[
|D_{j+1} - D_{j}| \leq L \text{ for all } j\ .
\] 
By the Azuma-Hoeffding inequality \cite{Azuma},
\begin{equation}
\Pro \bigg( | F(\bb 0, \bb z; \vec{\bb z}) - \E F(\bb 0, \bb z; \vec{\bb z}) | > s\bigg) \leq 2\exp\bigg(\frac{-s^2}{2|\bb z|_1L^2}\bigg)\ .
\end{equation}
The lemma follows by taking $s= t \sqrt{| \bb z |_1}$.   
\end{proof}

\subsection{Proof of Proposition~\ref{prop:pro1}}

We will first prove existence of $f$ and then we will prove properties (2)-(5).

\subsubsection{Existence of the limit}

As usual, the $L^1$ and almost sure convergence (to a finite limit) of $\lim_{n \to \infty} \frac{1}{n} F(\bb 0, n \bb x)$ for $\bb x \in \mathbb{Z}^d_+$ follows from Kingman's subadditive ergodic theorem. Because the model is not invariant under non-integer translations, to apply the same theorem with $\bb x \in \mathbb{R}^d_+$ we have to enlarge the space, as in \cite{Hoffman}. 

Let $\widetilde \Omega = [-1/2,1/2)^d \times \Omega$ and define a probability measure $\mathbb{\bar P}$ on this space as $m \times \mathbb{P}$, where $m$ is Lebesgue measure. We write a typical configuration in $\widetilde \Omega$ as $\widetilde \omega = (\bb r, \omega)$. For any $\bb y \in \mathbb{R}^d_+$, define the translation operator $\widetilde{T}_{\bb y}$ on $\widetilde \Omega$ in the following manner. If $\bb z \in \mathbb{R}^d_+$ then write $\overline{\bb z}$ for $\bb z - [\bb z]$. Then $\widetilde{T}_{\bb y}$ is defined as
\[
\widetilde{T}_{\bb y}(\bb r, \omega) = (\overline{\bb r + \bb y}, T_{[\bb r + \bb y]}\omega)\ ,
\]
$T_{[\bb r + \bb y]} \omega$ is the translation of $\omega$ by vertex $[\bb r + \bb y]$. Note that $\mathbb{\bar P}$ is invariant under $\widetilde{T}_{\bb y}$ (but not necessarily ergodic). Last, we define the free energy between vertices $\bb u$ and $\bb v$ in $(\bb r, \omega)$ as
\[
F(\bb u, \bb v)(\bb r, \omega) = F(\bb u, \bb v)(\omega)\ .
\]

Now that we set up the enlarged space, we briefly note that to show the existence of 
\begin{equation}\label{eq: existence}
f(\bb x) = \lim_{n \to \infty}(1/n) F(\bb 0, n \bb x)
\end{equation}
almost surely and in $L^1$, we may assume that all coordinates of $\bb x$ are strictly positive. Otherwise $\bb x$ is contained in a lower dimensional subspace of $\mathbb{R}^d$ and all directed paths from $\bb 0$ to $\bb x$ must stay in this subspace. By permutation invariance of the coordinates we could then assume the first $k$ coordinates of $\bb x$ are the nonzero ones and argue for the existence of $f$ as below in the space $\mathbb{R}^k_+$.

So fix $\bb x = (x_1, \ldots, x_d)$ with all coordinates nonzero and apply Kingman's subadditive ergodic theorem to the double sequence of variables $(F_{m,n})_{m \leq n}$ (for each ergodic component of the measure $\mathbb{\bar P}$) defined by
\[
F_{m,n}(\widetilde \omega) = F(m \bb x, n \bb x)(\widetilde \omega)\ .
\]
This provides the existence of the limit
\begin{equation}\label{eq: biggerconvergence}
(1/n) F(\bb 0, n \bb x) \to f(\bb x)(\widetilde \omega) <\infty \quad \mathbb{\bar P}\text{-almost surely} 
\end{equation}
and in $L^1(\mathbb{\bar P})$. We are left to argue that this implies convergence under the original measure and that this limit is almost surely constant.

We first address almost sure convergence. Equation \eqref{eq: biggerconvergence} means that if we select a point $\bb r$ uniformly at random in $[-1/2,1/2)^d$, then with probability one, $(1/n) F(\bb r, \bb r + n \bb x)$ converges for almost all $\omega$. Fix some such $\bb r$ and call this limit $f(\bb x)$. Because it does not depend on any finite number of edge weights, $f(\bb x)$ is constant $\mathbb{P}$-almost surely. Now write
\begin{eqnarray*}
|(1/n) F(\bb 0, n \bb x) - f(\bb x)| &\leq& |(1/n) F(\bb r, \bb r+ n \bb x) - f(\bb x)| \\
&+& (1/n) |F(\bb 0, n \bb x) - F(\bb r, n \bb x)| + (1/n) | F(\bb r, n \bb x) - F(\bb r, \bb r + n \bb x)| \ .
\end{eqnarray*}
By definition, $F(\bb r, n \bb x) = F(\bb 0, n \bb x)$, so we are left to show
\begin{equation}\label{eq: needtobound}
(1/n) |F(\bb 0, n \bb x) - F(\bb 0, \bb r + n \bb x)| \to 0 \text{ almost surely}\ .
\end{equation}

By the positivity of the $x_i$'s, fix $k\geq 1$ such that $(1/k) \leq \min_j x_j$. For such a choice, $$-1/2 + n x_j \geq 1/2 + (n-k)x_j$$ for all $j$ and $n \geq k$ and therefore
\[
\bb r + (n-k) \bb x \leq \bb n \bb x \leq \bb r + (n+k) \bb x \text{ for } n \geq k\ .
\]
By subadditivity,
\[
F(\bb 0, \bb r + (n+k) \bb x) -  F(n \bb x, \bb r + (n+k) \bb x) \leq F(\bb 0, n \bb x) \leq  F(\bb 0, \bb r + (n-k)\bb x) + F(\bb r + (n-k) \bb x, n \bb x)\ .
\]
Because $(1/n) F(\bb 0, \bb r + (n+k) \bb x)$ and $(1/n) F(\bb 0, \bb r + (n-k) \bb x)$ converge to the same number, we need then to show that
\[
(1/n) F(\bb r + (n-k) \bb x, n \bb x) \text{ and } (1/n) F(n \bb x, \bb r + (n+k)\bb x) \text{ converge to } 0\ .
\]
Translating both terms back by $(n-k) \bb x$ it suffices to show that for each $\e>0$ and $R>0$,
\[
\sum_n \mathbb{P}\left( \sup_{\stackrel{\bb t \leq \bb s}{|\bb t|_1, |\bb s |_1 \leq R}} F(\bb t, \bb s) \geq \e n\right) < \infty\ ,
\]
which follows because the supremum inside has finite mean. This proves almost sure existence of the limit \eqref{eq: existence}.

To show $L^1$ convergence, let \begin{equation}\label{eq:lpp}T(\bb 0, \bb x) = \max_{\gamma : \bb 0 \to \bb x} \tau(\gamma)\end{equation} and note the inequality
\[
0 \leq F(\bb 0, n \bb x) \leq T(\bb 0, n \bb x) + (n/\beta) |\bb x |_1 \log d\ .
\]
Because $(1/n) T(\bb 0, n \bb x)$ converges almost surely and in $L^1$ (see \cite[Proposition~2.1]{JMartin03}), the dominated convergence theorem finishes the proof.

\subsubsection{Properties of $f$}

We prove now that $f$ has the properties of Proposition~\ref{prop:pro1}. For item (2), the assumption $\mathbb{P}(\tau_e=0) < 1$ implies that $\mathbb{E} e^{-\beta \tau_e} <1$. So let $\bb x \in \mathbb{R}_+^d \setminus \{\bb 0\}$ and fix a directed path $\sigma : \bb 0 \to [\bb x]$:
\[
\mathbb{E} F(\bb 0, \bb x) = -\frac{1}{\beta} \mathbb{E} \log \frac{\sum_{\gamma : \bb 0 \to [\bb x]} \exp(-\beta \tau(\gamma))}{d^{|[\bb x]|_1}}  \geq -\frac{1}{\beta} \log \mathbb{E} e^{-\beta \tau(\sigma)} = -\frac{1}{\beta} \log \left( \mathbb{E} e^{-\beta \tau_e} \right)^{|[\bb x]|_1}\ .
\]
Here we have used Jensen's inequality with the logarithm. This implies
\[
f(\bb x) = \lim_{n \to \infty} \frac{1}{n} \mathbb{E} F(\bb 0, n \bb x) \geq -\frac{|\bb x|_1}{\beta} \log \mathbb{E} e^{-\beta \tau_e}\ ,
\]
giving \eqref{eq: infbound}.

Next, if $\lambda > 0$ then
\[
f(\lambda \bb x) = \lim_{n \to \infty} \frac{F(\bb 0, n \lambda \bb x)}{n} = \lambda \lim_{n \to \infty} \frac{F(\bb 0, n\lambda \bb x)}{n \lambda} = \lambda f( \bb x)\ ,
\]
proving item (3). (Here we have used that the convergence $(1/n)F(0,n\bb x)$ occurs over real $n$ going to infinity, which is a slight extension of part (1).) Items (4) and (5) follow immediately from the facts that $f$ is deterministic and $F$ is subadditive. This implies convexity of $f$: if $\bb x, \bb y \in \mathbb{R}^d_+$ and $\lambda \in [0,1]$,
\[
f(\lambda \bb x + (1-\lambda) \bb y) \leq \lambda f(\bb x) + (1-\lambda) f(\bb y)
\]
and therefore $f$ is continuous except possibly at the boundary of $\mathbb{R}^d_+$.

For the remainder of the section we prove continuity at the boundary using a direct adaptation of the arguments of \cite{JMartin03}. The strategy of the proof is to first consider the case where the weights are bounded and then use a truncation argument. The next lemma is the analogue of Lemma~3.2 in \cite{JMartin03}. 

\begin{lemma}\label{lem:lem127}Suppose  $\Pro(\tau_e \leq L) = 1$. Let $R>0$ and $\e>0$. There exists $\delta>0$ such that if $|\bb x| \leq R$ and $x_j = 0$ (where $1 \leq j \leq d$), then for all $0 \leq h \leq \delta$,
\[
|f(\bb x+h\bb e_j) - f(\bb x)|<\e\ .
\]
\end{lemma}

\begin{proof} By symmetry, we may take $j = 1$. We write a general vector in $\mathbb{R}^d_+$ as $(x,\bb x)$ where $x \in \mathbb{R}_+$ and $\bb x \in \mathbb{R}^{d-1}_+$.  We need to show that given $R>0$, for $\bb x=(x_2 ,x_3, \ldots ,x_d) \in \mathbb{R}^{d-1}_+$,
\[
f(h,\bb x) \rightarrow f(0,\bb x), \quad \text{as} \quad h \rightarrow 0^+\ ,
\]
uniformly in $\{ \bb x: |\bb x | \leq R\}$.

Let $\bb x$ and $h > 0$ be as above and $n \in \N$. A path from $\bb 0$ to the point $[n(h, \bb x)] $ contains exactly  $[nh]$  steps which increase the first coordinate, so can be decomposed into a concatenation of paths from $(r, \bb m_r)$ to $(r, \bb m_{r+1})$, $r = 0,1,2,\ldots, [nh]$, where $\bb m_r  \in \Z^{d-1}_+$ for each $r$ and 
\begin{equation}\label{eq:eq31}	
0= \bb m_0 \leq \bb m_1 \leq \cdots \leq \bb m_{[nh]+1} =  [n \bb x] \ .
\end{equation}

As noted in \cite{JMartin03}, the number of  the choices for the $\bb m_r$ satisfying the above equation is
\[
\prod_{i=2}^{d}\binom{[nx_i] + [nh]}{[nh]}\ .
\]
By StirlingÕs formula, this is $\exp[n \phi (h, \bb x) + o(n)]$, where 
\[
\phi(h,\bb x)=	\sum_{\stackrel{2\leq i \leq d}{x_i >0}} \bigg( h \log \frac{h+x_i}{h} + x_i \log \frac{x_i +h}{x_i} \bigg)\ .
\]
For each $0\leq i \leq  [nh]$  define $\bar F(\bb m_i, \bb m_{i+1})$ as the free energy of all paths joining $(i, \bb m_i)$ and $(i+1, \bb m_{i+1})$. We trivially have
\[
F (\bb 0, n(h, \bb x)) = -\frac{1}{\beta} \log \sum_{ \bb m_0, \bb m_1, \cdots, \bb m_{[nh] +1} } \bigg[ \prod_{i} \exp(-\beta \bar F(\bb m_i, \bb m_{i+1}))\bigg] \ . 
\]

 For fixed $\vec{\bb m} = \{\bb m_r \}$, by subadditivity and the definition of $f$
\begin{equation}\label{eq:eq33}
\E \sum_{i=0}^{[nh]} \bar F(\bb m_i, \bb m_{i+1}) \geq \E F(\bb 0,  n (0,\bb x)) \geq n f(0,\bb x)\ .
\end{equation}
We can now apply Lemma~\ref{lem:conc} to obtain the existence of $C_1>0$ such that for any $a >0$ 
\begin{equation*}
\Pro \left[\left|\sum_{i=0}^{[nh]}  \bar F(\bb m_i, \bb m_{i+1}) - \mathbb{E} \sum_{i=0}^{[nh]}  \bar F(\bb m_i, \bb m_{i+1})\right| \geq n a \right] \leq 2 \exp \left( -C_1 \frac{na^2}{L^2}\right).
\end{equation*}
Because $\phi(h,\bb x)$ tends to 0 uniformly in $|\bb x| \leq R$ as $h$ goes to zero, we can choose $\delta$ such that if $0 \leq h < \delta$ then
\[
\phi(h,\bb x) \leq \min \left\{ \frac{\beta \e}{2}, \frac{C_1\e^2}{18L^2} \right\}\ .
\]

Now, taking the sum over all possible $\vec{\bb m}$'s,
\begin{align*}
\Pro &\bigg[ F(\bb 0, n(h, \bb x)) \leq n f(0,\bb x) - n \e \bigg] \\
& \leq~ \exp(n \phi(h, \bb x) + o(n)) \max_{\vec{m}} \mathbb{P}\left( \sum_i \bar F(\bb{m}_i, \bb{m}_{i+1}) \leq nf(0,x)-n\e + \frac{n}{\beta}\phi(h,\bb x) + o(n)\right) \\
& \leq ~ \exp(n \phi(h, \bb x) + o(n)) \max_{\vec{m}} \mathbb{P}\left( \sum_i \bar F(\bb{m}_i, \bb{m}_{i+1}) - \sum_i \mathbb{E} \bar F(\bb{m}_i, \bb{m}_{i+1}) \leq -n\e/2 + o(n) \right)\ .
\end{align*}
For $n$ large, so that $o(n) \leq n\e/6$, we can apply the concentration inequality with $a = \epsilon/3$ to get an upper bound of
\[
2\exp(n \phi(h,\bb x)+o(n)) \exp\left( -C_1 \frac{n \e^2}{9L^2} \right)\ .
\]
By the choice of $h$, this is summable and therefore we can apply Borel-Cantelli to obtain
\[
f(0, \bb x) - f(h,\bb x) \leq \e\ .
\]

In the other direction, subadditivity implies 
\[
f(h, \bb x) - f(0,\bb x) \leq f(h, \bb 0) = h \E \tau_e \leq h L\ ,
\] 
which also tends to zero as $h$ goes to zero, uniformly over all $\bb x$.
\end{proof}

Now that we have established Lemma \ref{lem:lem127}, continuity of $f$ at the boundary of $\mathbb{R}^d_+$ follows immediately from the argument of \cite[Lemma~3.3]{JMartin03}.

\begin{lemma}Suppose  $\Pro(\tau_e \leq L) = 1$. Then $f$ is continuous on $\mathbb R^{d}_{+}$.

\end{lemma}
\begin{proof} 
The proof is identical to that of \cite[Lemma~3.3]{JMartin03}, replacing each instance of the word ``concave'' by ``convex."
\end{proof}

The next step is to show that one can remove the truncation and finally prove item~(5). For general weights $\tau_e$ we define the truncated ones $\tau_e^L = \min\{ \tau_e, L\}$. There is a corresponding free energy $F_L(\bb u, \bb v)$ for $\bb u \leq \bb v$ in $\mathbb{R}^d_+$ and limiting free energy $f_L(\bb u)$. Clearly
\[
F_L(\bb u, \bb v) \leq F(\bb u, \bb v) \text{ and so } f_L(\bb u) \leq f(\bb u)\ .
\]
The first part of the lemma says that $f_L \to f$ uniformly on compact subsets of $\mathbb{R}^d_+$, implying continuity for $f$. The second and third parts will be used later in the shape theorem.

\begin{lemma}[Truncation Lemma]\label{truncation}
Suppose that $\mathbb{E} \tau_e^{d+\alpha} < \infty$ for some $\alpha>0$. 
\begin{enumerate}
\item Given $R>0$ and $\e>0$ there exists $L$ such that
\[
\sup_{\stackrel{\bb u \in \mathbb{R}^d_+}{|\bb u|_1 \leq R}} \left( f(\bb u) - f_L(\bb u) \right) \leq \e\ .
\]
\item Given $\e>0$ there exists $L$ such that
\[
\mathbb{P}( F(\bb 0, \bb z) \leq F_L(\bb 0, \bb z) + \e|\bb z|_1 \text{ for all but finitely many } \bb z \in \mathbb{Z}^d_+ ) = 1\ .
\]
\item Given $\e>0$ there exists $L$ such that
\[
\mathbb{E} F(\bb 0, \bb z) \leq \mathbb{E} F_L(\bb 0, \bb z) + \e |\bb z|_1 \text{ for all } \bb z \in \mathbb{Z}^d_+\ .
\]
\end{enumerate}
\end{lemma}

\begin{proof}
We begin by estimating the difference between the free energies. This will be used in all parts of the lemma. For $\bb u \in \mathbb{R}^d_+$,
\begin{align}
F(\bb 0, \bb u) - F_L(\bb 0, \bb u) = -\frac{1}{\beta} \log \frac{\sum_{\gamma: \bb 0 \rightarrow [\bb u]} e^{-\beta \sum_{e \in \gamma} \tau_e}}{\sum_{\gamma: \bb 0 \rightarrow [\bb u]} e^{-\beta \sum_{e \in \gamma} \tau_e^L}} &= -\frac{1}{\beta} \log \frac{\sum_{\gamma: \bb 0 \rightarrow [\bb u]} e^{-\beta \sum_{e \in \gamma} \tau_e - \tau_e^L -\beta \sum_{e \in \gamma}  \tau_e^L}}{\sum_{\gamma: \bb 0 \rightarrow [\bb u]} e^{-\beta \sum_{e \in \gamma} \tau_e^L}} \nonumber \\ 
 &\leq \max_{\gamma: \bb 0 \rightarrow [\bb u]} \sum_{e \in \gamma} (\tau_e - \tau^L_e) \label{sbound}\ .
\end{align}

The last term in \eqref{sbound} is just the last-passage time (see \eqref{eq:lpp}) $\tilde T_L(\bb 0, [\bb u])$ from $\bb 0$ to $[\bb u]$ using i.i.d. edge weights $(\tilde \tau_e)$ whose distribution satisfies
\[
\tilde \tau_e = \begin{cases}
0 & \text{with probability } \mathbb{P}(\tau_e \leq L) \\
\tau_e-L & \text{with probability } \mathbb{P}(\tau_e > L)
\end{cases}\ .
\]
Since $\mathbb{E} \tilde \tau_e <\infty$, \cite[Proposition~2.2]{JMartin03} implies that the limit shape function
\[
0 \leq G_L(\bb u) := \lim_{n\rightarrow \infty} (1/n) \tilde T_L(\bb 0, n\bb u) < \infty
\]
exists a.s. and in $L^1$. Furthermore, \cite[Lemma~3.5(i)]{JMartin03} provides a constant $c>0$ such that
\begin{equation}\label{eq: name}
\text{ for all } \bb z \in \mathbb{Z}^d_+,~ \mathbb{E} \tilde T_L(\bb 0, \bb z) \leq c |\bb z|_1 \int_0^\infty  \mathbb{P}(\tilde \tau_e \geq s)^{1/d} ~\text{d} s\ .
\end{equation}
The condition $\mathbb{E} \tau_e^{d+\alpha} < \infty$ implies that the integral on the right is finite. Given $\e, R>0$, choose $L$ such that $\int_0^\infty \mathbb{P}(\tilde \tau_e \geq s)^{1/d} ~\text{d}s < \e/(cR)$. Then for any $\bb u \in \mathbb{R}^d_+$ such that $|\bb u|_1 \leq R$,
\[
G_L(\bb u) = \lim_{n \to \infty} (1/n) \mathbb{E} \tilde T_L(\bb 0,n \bb u) \leq \e\ .
\]
Therefore, $f(\bb u) - f_L(\bb u) \leq G_L(\bb u) \leq \e$, proving part 1.

For the second part, use \eqref{eq: name} to choose $L$ large enough that $G_L(\bb u) \leq \frac{\e}{2} |\bb u|_1$ for all $\bb u \in \mathbb{R}_+^d$. The shape theorem in last-passage percolation \cite[Theorem 5.1]{JMartin03} implies that for all but finitely many $ \bb u \in \Z^d_{+}$,
\begin{equation}\label{das}
\left|\max_{\gamma: \bb 0 \rightarrow \bb u} \sum_{e \in \gamma}( \tau_e - \tau^L_e) - G_L(\bb u) \right| \leq \frac{\e}{2} |\bb u|_1\ .
\end{equation}
Combining \eqref{sbound}, $G_L(\bb u) \leq \frac{\e}{2} |\bb u|_1$ and \eqref{das}, we end the proof of part two.

Part three also follows from \eqref{eq: name}; given $\e>0$ we can find $L$ such that for all $\bb z \in \mathbb{Z}^d_+$, $\mathbb{E} \tilde T_L(\bb 0, \bb z) \leq \e |\bb z|_1$. Taking expectation in \eqref{sbound} 
and combining with this statement finishes the proof.
\end{proof}

\subsection{Proof of Proposition~\ref{shapethm}}
If $\mathbb{P}(\tau_e=0)=1$ then the model is deterministic and there is nothing to prove. Otherwise we use part (ii) of Proposition~\ref{prop:pro1} and subadditivity to get
\[
0 < \inf_{ \bb 0 \neq \bb x \in \mathbb{R}^d_+} \frac{f(\bb x)}{|\bb x|_1} \leq \sup_{\bb 0 \neq \bb x \in \mathbb R^d_{+}} \frac{f(\bb x)}{| \bb x|_1} \leq d ~f(1,0,\ldots,0) = d ~\E \tau_e < \infty\ .
\]
Therefore to prove the shape theorem we must show the following. For any $\e > 0$, there are almost surely only finitely many $\bb z \in \Z^d_+$ such that 
\[
|F(\bb 0, \bb z) - f(\bb z)| \geq  \e |\bb z|_1\ .
\]
This statement is a consequence of the following lemmas:

\begin{lemma}\label{lem:lem1ls} 
For each $\e>0$, 
\[
\mathbb{P}\left( |F(\bb 0, \bb z) - \E F(\bb 0, \bb z)| < \e | \bb z |_1 \text{ for all but finitely many } \bb z \in \Z^d_+\right) = 1\ .
\]
\end{lemma}

\begin{lemma}\label{lem:lem2ls}
For each $\e >0$, for all but finitely many  $\bb z \in \Z^d_+$, $| \E F(\bb 0, \bb z) -f(\bb z)| < \e | \bb z |_1$.
\end{lemma}

\begin{proof}[Proof of Lemma \ref{lem:lem1ls}] If the weights are bounded by $L>0$ then one can apply the concentration inequality of Lemma \ref{lem:conc} to obtain:

\begin{equation}
\Pro \bigg( |F(\bb 0, \bb z) - \E F(\bb 0, \bb z)| > \e | \bb z|_1\bigg) \leq 2 \exp \left( - \frac{\e^2| \bb z|_1}{2L^2} \right)\ .
\end{equation}

For $n \in \mathbb{N}$, there are no more than $C (n+1)^d$ points $\bb z$ such that $| \bb z |_1 = n$. Thus, 
\[
\sum_{\bb z \in \Z^d_+} \Pro \bigg( |F(\bb 0, \bb z) - \E F(\bb 0, \bb z)| > \e |\bb z|_1 \bigg) \leq 2C \sum_{n \in \N} (n+1)^d \exp\bigg ( - \frac{\e^2n}{2L^2}\bigg) < \infty
\]
and Borel-Cantelli finishes the proof in the case of bounded weights.

The case of unbounded weights now follows by combining the above result with parts 2 and 3 of the truncation lemma.
\end{proof}

\begin{proof}[Proof of Lemma \ref{lem:lem2ls}] 
From subadditivity, $\E F(\bb 0, \bb z) \geq f(\bb z)$ for all $\bb z$. Therefore we just need to show that if $\e>0$ then $ \E F(\bb 0, \bb z) < f(\bb z)+\e | \bb z|_1$ except for finitely many $\bb z$. 

First, assume that the weights are bounded by $L>0$. Fix $a > 0$. By Proposition \ref{prop:pro1}, part (6), $f$~is continuous on $\R^d_{+}$, and hence is uniformly continuous on the compact subset $\{\bb x \in \R^d_+ : |\bb x|_1 \leq 2d \}.$ Choose $0 < u < \min(1,a)$ such that 
\[
\text{whenever } |\bb x|_1 \leq d \text{ and } | \bb x - \bb x' |_1 \leq ud,~ |f(\bb x)-f(\bb x')|\leq a\ .
\]
Now let
\[
\mathcal{C} = \bigg \{ u \bb r , \bb r \in \left\{0, 1, \ldots, \left\lfloor \frac{1}{u} \right\rfloor \right\}^d \bigg\}\ .
\]
$\mathcal{C}$ is a finite subset of $\R^d_+$ and for each $\bb y \in \mathcal{C}$, we have (by Proposition \ref{prop:pro1} part (1)), 
\[
\frac{1}{n}\E F(\bb 0,  n \bb y) \rightarrow f(\bb y), \quad \text{as} \quad n \rightarrow \infty\ .
\]
Hence there is $N=N(a)$ such that, for all $n\geq N$ and all $\bb y\in \mathcal{C}$,
\[
\E F(\bb 0, n \bb y) \leq n (f(\bb y)+a)\ .
\]
Let $\bb z = (z_1, \ldots, z_d)$ in $\mathbb{Z}_+^d$ satisfy $\max z_i \geq N$. Define 
\[
\bb y=u \left( \left\lfloor \frac{z_1}{u \max z_i}\right\rfloor , \ldots, \left\lfloor \frac{z_d}{u \max z_i} \right\rfloor \right)\ .
\]
Then $\bb y \in \mathcal{C}$, with $(\max z_i) \bb y \leq \bb z$, with $| \bb y |_1 \leq d$ and with
\[
\left | \frac{\bb z}{\max z_i} - \bb y \right |_1 \leq ud \leq ad\ .
\]

Using first subadditivity, the bound $\tau_e \leq L$, then the continuity bounds above, we obtain

\begin{equation*}
\begin{split}
 \E F (\bb 0,\bb z) &\leq  \E F (\bb 0, (\max z_i )\bb y ) + \E F (\bb0, \bb z -  (\max z_i )\bb y ) \\
&\leq \E F (\bb 0, (\max z_i )\bb y ) + \left(L+ \frac{\log d}{\beta}\right) |[\bb z - (\max z_i )\bb y]|_1\\
&\leq (f(\bb y)+a)(\max z_i) + \left(L + \frac{\log d}{\beta}\right) (|\bb z - (\max z_i )\bb y|_1 + d)\\
&\leq f(\bb z) + (\max z_i ) \bigg(2a + \left(L + \frac{\log d}{\beta}\right) \left | \frac{\bb z}{\max z_i} - \bb y \right|_1+ \left(L + \frac{\log d}{\beta}\right)\frac{d}{\max z_i}\bigg)\\
 &\leq f(\bb z) + (\max z_i) \bigg( 2a + \left(L + \frac{\log d}{\beta}\right)ad + \left(L + \frac{\log d}{\beta}\right)\frac{d}{\max z_i}\bigg)\ .
 \end{split}
\end{equation*} 
 
Hence if $a < \e (4(2 + (L+ (1/\beta)\log d)d))^{-1}$, then for all $\bb z$ with $|\bb z|_1 \geq \max (N(a) ,2(L+(1/\beta)\log d)d/ \e)$, we have
\[
\E F (\bb 0,\bb z) \leq f(\bb z) + \e |\bb z|_1\ ,
\]
and this finishes the proof in the case of bounded weights. 

The case of unbounded weights follows from parts 1 and 3 of the truncation lemma. Indeed, given $\e>0$, part 1 gives $L$ such that for all $\bb u \in \mathbb{R}^d_+$ with $|\bb u|_1 \leq 1$, $f(\bb u) - f_L(\bb u) < \e/3$. Then as both limiting free energies are positive homogeneous,
\[
f(\bb u) - f_L(\bb u) < (\e/3) |\bb u|_1 \text{ for all } \bb u \in \mathbb{R}^d_+\ .
\]
Now part 3 provides a (possibly larger) $L$ such that also
\[
\mathbb{E} F(\bb 0, \bb u) - \mathbb{E} F_L(\bb 0, \bb u) \leq (\e/3) |\bb u|_1 \text{ for all } \bb u \in \mathbb{R}^d_+
\]
By combining these with the first part of this proof, we are done.

\end{proof}

\section{Alexander's method}\label{Alexander}

The goal of this last section is to prove the following lemma, which is based entirely on work of Alexander \cite{Alexander} and the extension by Chatterjee \cite{Sourav}.

\begin{lemma}\label{lem: AC}
Given $\chi'>\chi_a$ there exists $\alpha>0$ such that
\[
\sup_{ \bb x \in \mathbb{Z}^d_+\setminus \{ \bb 0\}} \E \exp\left( \alpha \frac{|F(\bb 0,\bb x)-f(\bb x)|}{|\bb x|_1^{\chi'}} \right) < \infty\ .
\]
\end{lemma}

The main task in proving Lemma~\ref{lem: AC} is to control the order of deviations of $h(\bb x) := \E F(\bb 0,\bb x)$ from $f(\bb x)$. In the zero temperature case this was beautifully done by Alexander in \cite{Alexander} and adapted by Chatterjee in \cite{Sourav}. Recently, in the positive temperature case, Alexander and Zygouras in \cite{AlexZ} showed that $\E F(\bb 0,\bb x) - f(\bb x) = O(\frac{|\bb x|^{\frac{1}{2}}}{ \log |\bb x |})$  under a certain assumption on the weight distribution. Since we take $\chi' > \chi_a$ and therefore do not require a fine result involving logarithms, we do not need to use the methods developed in \cite{AlexZ}. 

We set $H_{\bb x}$ to be any hyperplane tangent to $f(\bb x)B$ at $\bb x$.  Let $H_{\bb x}^0$ be the translation of $H_{\bb x}$ that passes through the origin. There exists a unique linear functional $f_{\bb x}$ on $\R^d$ satisfying $f_{\bb x}(\bb y) = 0$ for all $\bb y \in H_{\bb x}^0$ and $f_{\bb x}(\bb x) = f(\bb x)$. Note that $f_{\bb x}(\bb y) \leq f(\bb y)$ for all $\bb y$. We can see this as follows. If $\bb y = 0$ it is clearly true. Otherwise, $\bb y / f(\bb y) \in B$ and so $f_{\bb x}(\bb y / f(\bb y)) \leq 1$.  Furthermore, since $f$ is convex and symmetric about the diagonal through $\bb 0$ and $\bb e$, we have 
\begin{equation}\label{Alternative1}
f(\bb z) \geq \frac{f(\bb e)|\bb z|_1}{d}\ .
\end{equation} 
From subadditivity and symmetry we also obtain 
\begin{equation}\label{Alternative2}
f(\bb z) \leq f(\bb e_1) |\bb z|_1\ .
\end{equation}

Fix $\chi''>\chi_a$. For each $\bb x \in \R_+^d, C >0$ and $K>0$ define

\[Q_{\bb x}(C,K) := \{ \bb y \in \Z_+^d: |\bb y|_1\leq K|\bb x|_1, f_{\bb x}(\bb y) \leq f(\bb x), h(\bb y) \leq f_{\bb x}(\bb y) + C|\bb x|_1^{\chi''} \}\ ,
\]
\[ G_{\bb x} := \{ \bb y \in \Z_+^d : f_{\bb x}(\bb y)> f(\bb x) \}  \ , \]
\[ \Delta_{\bb x} := \{\bb y \in Q_{\bb x}: \bb y \; \text{adjacent to} \; \Z^d\setminus Q_{\bb x},~ \bb y \text{ not adjacent to } G_{\bb x}\;    \} \ , \]
\[ D_{\bb x} := \{ \bb y \in Q_{\bb x}: \bb y \; \text{adjacent to} \; G_{\bb x} \} \ .  \] 
Now set $$Q_{\bb x} = Q_{\bb x} (C_1, 2 d^{3/2}f(\bb e_1)/ f(\bb e)+1)$$
where $C_1 := 320 d^2/\alpha.$ 
The following lemma is the analogue of \cite[Lemma~3.3]{Alexander} and \cite[Lemma~4.3]{Sourav}:

\begin{lemma}\label{lemma1111} There exists a constant $C'>0$ such that if $|\bb x |_1> C'$ then the following hold.
\begin{enumerate}
\item If $\bb y \in Q_{\bb x}$ then  $f(\bb y) \leq 2 f(\bb x)$ and $|\bb y |_1 \leq 2d^{\frac{3}{2}}f(\bb e_1)|\bb x|_1 / f(\bb e)$.
\item If $\bb y \in \Delta_{\bb x}$ then $h(\bb y) - f_{\bb x}(\bb y) \geq C_1 |\bb x |_1^{\chi''}(\log |\bb x|_1)/2.$
\item If $\bb y \in D_{\bb x}$ then $f_{\bb x}(\bb y) \geq 5g(\bb x)/6$.
\end{enumerate}
\end{lemma}
\begin{proof}
The proof is as in \cite[Lemma 4.3]{Sourav} where equations \eqref{Alternative1} and \eqref{Alternative2} replace equation (11), which is not necessarily true in the model considered here.
\end{proof}

Although in the last lemma we had to use equations \eqref{Alternative1} and \eqref{Alternative2} to adapt the proof of Lemma \ref{lemma1111}, the next result follows directly from \cite[Lemma 1.6]{Alexander} (or \cite[Lemma 4.2]{Sourav}). In fact, those are undirected results, but the directed version follows as in \cite[Section~4]{Alexander}. 

\begin{lemma}\label{AlexL} Suppose that for some $M>1$, $C>0, K>0$ and $a>1$ the following holds. For each $\bb x\in \Z^d_+$ with $|\bb x|_1 \geq M$, there exists an integer $n \geq 1$, a directed lattice path $\gamma$ from $\bb 0$ to $n \bb x$ and a sequence of sites $\bb 0=v_0 \leq v_1 \leq \ldots \leq v_m=n\bb x$ in $\gamma$ such that $m \leq an$ and $v_i - v_{i-1} \in Q_{\bb x}(C,K)$ for all $1\leq i \leq m$. Then for some $C'>0$ and for all $\bb x \in \Z^d_+$ we have

\[ f(\bb x) \leq h(\bb x) \leq  f(\bb x) + C' |\bb x|_1^{\chi''} \log |\bb x|_1 \ .\]
\end{lemma}

We now check that the assumption on the existence of the exponent $\chi_a$ implies that the hypothesis of Lemma \ref{AlexL} is satisfied with the choices $C=C_1$, $K= 2d^{\frac{3}{2}}f(\bb e_1)/f(\bb e)$ and $M$ large enough. We will need more notation though.

A collection of vertices $(v_i)$, $i=0,\ldots, m$ satisfying the hypothesis of Lemma \ref{AlexL} is called a skeleton of $n\bb x$ with $m+1$ steps. Let $\mathcal S_m $ be the collection of all possible skeletons of $n\bb x$ with $m+1$ steps. That is, define  $$ \mathcal S_m = \{\vec{\bb v} : \vec{\bb v} = \{0 \leq v_1\leq  v_2 \leq \ldots \leq   v_m\}\; \text{with} \; v_{i+1}-v_i \in Q_x(C,K) \ \forall i = 0, \ldots, m-1\} \ .$$ By Lemma~\ref{lemma1111}, part 1, there exists a constant $C_0$ such that the cardinality of $\mathcal S_m$ satisfies 
\begin{equation}\label{skeleton}
|\mathcal S_m | \leq (C_0|\bb x|_1^d)^m \ .
\end{equation}

Given a skeleton  $ \vec{\bb v} $, $F(v_i, v_{i+1})$, $i=0, \ldots, m-1$ are independent random variables. Also, by Definition~\ref{def3} and Lemma \ref{lemma1111}, part 1, there exists $C_1>0$ such that for all $i$,
\[ \E \exp\bigg(\frac{\alpha}{K^{\chi''}} \frac{|F(v_i, v_{i+1})-\E F(v_i, v_{i+1})|}{|\bb x|_1^{\chi''}} \bigg) < C_1 \ .\]
Therefore, for all $t\geq0$ 

\begin{equation}\label{sas112}
\Pro \bigg( \sum_{i=0}^{m-1} |F(v_i, v_{i+1})-\E F(v_i, v_{i+1})|  \geq t\bigg) \leq \exp \bigg( -\frac{\alpha t}{(K|\bb x|_1)^{\chi''}} \bigg) C_1^{m}\ .  
\end{equation}

Choosing $t= C_2m|\bb x|_1^{\chi''} \log |\bb x|_1$ for $C_2$ large enough, a simple union bound combining \eqref{skeleton} with \eqref{sas112} implies that there exist constants $C_3$ and $C_4>0$ such that if $|\bb x|_1\geq C_3$ then

$$ \Pro \bigg( \exists \;  \vec{\mathbf v} \in \mathcal S_m \; \text{such that} \;  \sum_{i=0}^{m-1} |F(v_i, v_{i+1})-\E F(v_i, v_{i+1})|  \geq C_2m|\bb x|_1^{\chi''} \log |\bb x|_1 \bigg) \leq e^{-C_4 m \log |\bb x|_1} \ .$$
This however implies that $|\bb x|_1$ bigger than some $C_5$,
\begin{align}
\Pro \bigg( \exists \; m \geq 1, \;\; \vec{\mathbf v} \in \mathcal S_m \; \text{such that} \;  \sum_{i=0}^{m-1} |F(v_i, v_{i+1})-\E F(v_i, v_{i+1})|  &\geq C_2m|\bb x|_1^{\chi''} \log |\bb x|_1 \bigg) \nonumber \\
 &\leq (1/2) e^{-C_4 m \log |\bb x|_1} \label{eq:chalk}\ .
\end{align}

Once equation \eqref{eq:chalk} is established one can follow the same lines as in the proof of \cite[Proposition~3.4]{Alexander} to show that the hypothesis of Lemma~\ref{AlexL} is satisfied. Namely, we obtain:

\begin{lemma} There exists a constant $C$ such that if $|\bb x |_1 \geq C$ then for sufficiently large $n$ there exists a directed lattice path from $\bb 0$ to $n\bb x$ with a skeleton of $2n+1$ or fewer vertices.
\end{lemma}

We finish this section with the proof of Lemma \ref{lem: AC}.

\begin{proof}[Proof of Lemma \ref{lem: AC}] Given $\chi' >\chi_a$, let $\chi''$ be such that $\chi'>\chi''>\chi_a$. Taking $\alpha$ as in Definition~\ref{def3}, Lemma \ref{AlexL} (applied to $\chi''$) combined with the triangle inequality implies the existence of $C,C'>0$ such that for all $\bb x \in \Z^d_{+} \setminus \{\bb 0\}$, 

\[\E \exp \bigg(\alpha \frac{|F(\bb 0,\bb x)-f(\bb x)|}{|\bb x |_1^{\chi'}}\bigg) \leq C \E \exp \bigg( \alpha \frac{|F(\bb 0,\bb x)-\E F(\bb 0, \bb x)|}{|\bb x |_1^{\chi'}}\bigg) \ < C'\ .\]

\end{proof}

\bigskip
\noindent
{\bf Acknowledgements.} We thank K. Alexander and N. Zygouras for the explaining their recent results and the method used in the proof of Lemma~\ref{lem: AC}. We also thank C. Newman for suggesting the idea to extend the results of \cite{AD2} to positive temperature. A. A. thanks the Courant Institute for hospitality during visits while some of this work was done. M. D. thanks the Courant Institute and C. Newman for summer funds and support. Last, we thank L.-P. Arguin for telling us about Skype's ``screen share'' and so helping us to complete this work.


\thebibliography{1}

\bibitem{AKQ} T. Alberts, K. Khanin and J. Quastel. (2012). Intermediate Disorder Regime for 1+1 Dimensional Directed Polymers. preprint. {\it arXiv:1202.4398}




\bibitem{AlexZ} K. S. Alexander and N. Zygouras.  (2012). Subgaussian concentration and rates of convergence in directed polymers. preprint. {\it arXiv:1204.1819}

\bibitem{Alexander} K. S. Alexander (1997). Approximation of subadditive functions and convergence rates in limiting-shape results.  {\it Ann. Probab.} {\bf 25} 30--55.

\bibitem{ABC} A. Auffinger, J. Baik and I. Corwin. (2012). Universality for directed polymers in thin rectangles. preprint. {\it arXiv:1204.4445}


\bibitem{AD2}
A. Auffinger and M. Damron. (2011). A simplified proof of the relation between scaling exponents in first-passage percolation. preprint. {\it arXiv:1109.0523}



\bibitem{Azuma} K. Azuma. (1967). Weighted sums of certain dependent random variables. {\it Tohoku Math. J.} {\bf 19}, no. 3, 357--367. 



\bibitem{BCD}
A. Borodin,  I. Corwin and D. Remenik. (2012). Log-gamma polymer free energy fluctuations via a fredholm identity. preprint. {\it arXiv:1206.4573}


\bibitem{Sourav}
S. Chatterjee. (2011). The universal relation between scaling exponents in first-passage percolation. To appear in {\it Annals of Math.}

\bibitem{Corwin} I. Corwin. (2011). The Kardar-Parisi-Zhang equation and universality class. (2012). {\it Random Matrices: Theory and Applications}, {\bf 1}.

\bibitem{CDurrett}
T. Cox and R. Durrett. (1981). Some limit theorems for percolation with necessary and sufficient conditions. {\it Ann. Probab.} {\bf 9} 583--603.

\bibitem{CKesten}
T. Cox and  H. Kesten. (1981). On the continuity of the time constant of first-passage percolation.
{\it J. Appl. Probab.} {\bf 18}, no. 4, 809--819.

\bibitem{Hoffman} C. Hoffman. (2008). Geodesics in first passage percolation. {\it Ann. Appl. Probab.} {\bf 18} 1944--1969.

\bibitem{Joha} K. Johansson. (2000). Shape fluctuations and random matrices. 
{\it Comm. Math. Phys.} {\bf 209}, no. 2, 437--476. 

\bibitem{Joha2} K. Johansson. (2001). Transversal fluctuations for increasing subsequences on the plane
{\it Probab. Theory and Related Fields.} {\bf 116}, no. 4, 445--456.

\bibitem{KPZ}
M. Kardar, G. Parisi and Y. Zhang. (1986). Dynamic scaling of growing interfaces. {\it Phys. Rev. Lett.} {\bf 56} 889--892.

\bibitem{Kesten}
H. Kesten. Aspects of first-passage percolation. {\it \'Ecole d'\'et\'e de probabilit\'es de Saint-Flour, XIV--1984,} 125--264, Lecture Notes in Math., 1180, {\it Springer, Berlin}, 1986.

\bibitem{Kesten93}
H. Kesten. (1993). On the speed of convergence in first-passage percolation. {\it Ann. Appl. Probab.} {\bf 3} 296--338.

\bibitem{JMartin03}
J. Martin. (2004). Limiting shape for directed percolation models. {\it Ann. Probab.} {\bf 32} 2908--2937.

\bibitem{NP}
C. Newman and M. Piza. (1995). Divergence of shape fluctuations in two dimensions. {\it Ann. Probab.} {\bf 23} 977--1005.

\bibitem{Piza}
M. Piza. (1997). Directed polymers in a random environment: Some results on fluctuations.  {\it J. Stat. Phys.} {\bf 89} 581--603.

\bibitem{Richardson}
D. Richardson. (1973). Random growth in a tessellation. {\it Proc. Cambridge Philos. Soc.} {\bf 74} 515--528.

\bibitem{Sepa} T. Sepp{\"a}l{\"a}inen. (2012). Scaling for a one-dimensional directed polymer with boundary conditions. {\it Ann. Probab.}, {\bf 40} 19--73.
\end{document}